\documentclass{article}
\usepackage{amsmath}
\usepackage{amssymb}

\newtheorem{theorem}{Theorem}

\newtheorem{corollary}[theorem]{Corollary}

\newtheorem{lemma}[theorem]{Lemma}

\newenvironment{proof}[1][Proof]{\noindent\textbf{#1.} }{\ \rule{0.5em}{0.5em}}
\input{tcilatex}

\begin{document}

\title{Every finitely generated two-sided ideal of a Leavitt path algebra is
a principal ideal }
\author{Kulumani M. Rangaswamy \\
Department of Mathematics, University of Colorado\\
Colorado Springs, Colorado 80918, USA\\
E-mail: krangasw@uccs.edu}
\maketitle

\begin{abstract}
Let $E$ be an arbitrary graph and $K$ be any field. For every non-graded
ideal $I$ of the Leavitt path algebra $L_{K}(E)$, we give an explicit
description of the generators of $I$. Using this, we show that every
finitely generated ideal of $L_{K}(E)$ must be principal. In particular, if $%
E$ is a finite graph, then every ideal of $L_{K}(E)$ must be principal ideal.
\end{abstract}

\section{Introduction}

The notion of Leavitt path algebras of a graph $E$ was introduced and
initially studied in \cite{AA}, \cite{AMP} as algebraic analogues of C$%
^{\ast }$-algebras and the analysis of the structure of their two-sided
ideals has received much attention in recent years. For instance, Tomforde 
\cite{Tomforde 20} described all the graded ideals in a Leavitt path algebra
in terms of their generators. In \cite{C} and \cite{ABCR} generating sets
for arbitrary ideals of a Leavitt path algebra were established while in 
\cite{APS} and \cite{R} the prime ideal structure of a Leavitt path algebra
was described. In this note, complementing Tomforde's theorem on graded
idreals, we first give an explicit description of a set of generators for
non-graded ideals in the Leavitt path algebra $L_{K}(E)$ of an arbitrary
graph $E$ over a field $K$. Using this we prove that every finitely
generated ideal in $L_{K}(E)$ must be a principal ideal. As a corollary, we
show that if $E$ is a finite graph, then every ideal of $L_{K}(E)$ must be a
principal ideal. The method involves a judicious selection of finitely many
mutually orthogonal generators to replace a given finite set of generators
of the ideal $I$. The sum of these orthogonal generators will then be the
desired single generator for $I$.

\section{Preliminaries}

All the graphs $E$ that we consider here are arbitrary in the sense that no
restriction is placed either on the number of vertices in $E$ (such as being
a countable graph) or on the number of edges emitted by any vertex (such as
being a row-finite graph). \ We shall follow \cite{ABCR}, \cite{R} for the
general notation, terminology and results. For the sake of completeness, we
shall outline some of the concepts and results that we will be using.

A (directed) graph $E=(E^{0},E^{1},r,s)$ consists of two sets $E^{0}$ and $%
E^{1}$ together with maps $r,s:E^{1}\rightarrow E^{0}$. The elements of $%
E^{0}$ are called \textit{vertices} and the elements of $E^{1}$ \textit{edges%
}. If $s^{-1}(v)$ is a finite set for every $v\in E^{0}$, then the graph is
called \textit{row-finite}.

If a vertex $v$ emits no edges, that is, if $s^{-1}(v)$ is empty, then $v$
is called a\textit{\ sink}. A vertex $v$ is called an \textit{infinite
emitter} if $s^{-1}(v)$ is an infinite set, and $v$ is called a \textit{%
regular vertex} if $s^{-1}(v)$ is a finite non-empty set. A path $\mu $ in a
graph $E$ is a finite sequence of edges $\mu =e_{1}\dots e_{n}$ such that $%
r(e_{i})=s(e_{i+1})$ for $i=1,\dots ,n-1$. In this case, $n$ is the length
of $\mu $; we view the elements of $E^{0}$ as paths of length $0$. We denote
by $\mu ^{0}$ the set of vertices of the path $\mu $, i.e., the set $%
\{s(e_{1}),r(e_{1}),\dots ,r(e_{n})\}$.

A path $\mu $ $=e_{1}\dots e_{n}$ is \textit{closed} if $r(e_{n})=s(e_{1})$,
in which case $\mu $ is said to be based at the vertex $s(e_{1})$. A closed
path $\mu $ as above is called \textit{simple} provided it does not pass
through its base more than once, i.e., $s(e_{i})\neq s(e_{1})$ for all $%
i=2,...,n$. The closed path $\mu $ is called a \textit{cycle} if it does not
pass through any of its vertices twice, that is, if $s(e_{i})\neq s(e_{j})$
for every $i\neq j$. An \textit{exit }for a path $\mu =e_{1}\dots e_{n}$ is
an edge $e$ such that $s(e)=s(e_{i})$ for some $i$ and $e\neq e_{i}$. We say
that $E$ satisfies \textit{Condition }(L) if every simple closed path in $E$
has an exit, or, equivalently, every cycle in $E$ has an exit. A graph $E$
is said to satisfy \textit{Condition }(K\textit{)} provided no vertex $v\in
E^{0}$ is the base of precisely one simple closed path, i.e., either no
simple closed path is based at $v$, or at least two are based at $v$.

We define a relation $\geq $ on $E^{0}$ by setting $v\geq w$ if there exists
a path in $E$ from $v$ to $w$. A subset $H$ of $E^{0}$ is called \textit{%
hereditary} if $v\geq w$ and $v\in H$ imply $w\in H$. A hereditary set is 
\textit{saturated} if, for any regular vertex $v$, $r(s^{-1}(v))\subseteq H$
implies $v\in H$.

For each $e\in E^{1}$, we call $e^{\ast }$ a ghost edge. We let $r(e^{\ast
}) $ denote $s(e)$, and we let $s(e^{\ast })$ denote $r(e)$.

Given an arbitrary graph $E$ and a field $K$, the \textit{Leavitt path }$K$%
\textit{-algebra }$L_{K}(E)$ is defined to be the $K$-algebra generated by a
set $\{v:v\in E^{0}\}$ of pairwise orthogonal idempotents together with a
set of variables $\{e,e^{\ast }:e\in E^{1}\}$ which satisfy the following
conditions:

(1) \ $s(e)e=e=er(e)$ for all $e\in E^{1}$.

(2) $r(e)e^{\ast }=e^{\ast }=e^{\ast }s(e)$\ for all $e\in E^{1}$.

(3) (The "CK-1 relations") For all $e,f\in E^{1}$, $e^{\ast}e=r(e)$ and $%
e^{\ast}f=0$ if $e\neq f$.

(4) (The "CK-2 relations") For every regular vertex $v\in E^{0}$, 
\begin{equation*}
v=\sum_{e\in E^{1},s(e)=v}ee^{\ast }.
\end{equation*}

If $\mu =e_{1}\dots e_{n}$ is a path in $E$, we denote by $\mu ^{\ast }$ the
element $e_{n}^{\ast }\dots e_{1}^{\ast }$ of $L_{K}(E)$.

A useful observation is that every element $a$ of $L_{K}(E)$ can be written
as $a=\tsum\limits_{i=1}^{n}k_{i}\alpha _{i}\beta _{i}^{\ast }$, where $%
k_{i}\in K$, $\alpha _{i},\beta _{i}$ are paths in $E$ and $n$ is a suitable
integer (see \cite{AA}).

The following concepts and results from \cite{Tomforde 20} will be used in
the sequel. A vertex $w$ is called a \textit{breaking vertex }of a
hereditary saturated subset $H$ if $w\in E^{0}\backslash H$ is an infinite
emitter with the property that $1\leq |s^{-1}(v)\cap r^{-1}(E^{0}\backslash
H)|<\infty $. The set of all breaking vertices of $H$ is denoted by $B_{H}$.
For any $v\in B_{H}$, $v^{H}$ denotes the element $v-\sum_{s(e)=v,r(e)\notin
H}ee^{\ast }$. Given a hereditary saturated subset $H$ and a subset $%
S\subseteq B_{H}$, $(H,S)$ is called an \textit{admissible pair }and $%
I_{(H,S)}$ denotes the ideal generated by $H\cup \{v^{H}:v\in S\}$. It was
shown in \cite{Tomforde 20} that the graded ideals of $L_{K}(E)$ are
precisely the ideals of the form $I_{(H,S)}$ for some admissible pair $(H,S)$%
. Moreover, it was shown that $I_{(H,S)}\cap E^{0}=H$ and $\{v\in
B_{H}:v^{H}\in I_{(H,S)}\}=S$.

Given an admissible pair $(H,S)$, the corresponding \textit{quotient graph} $%
E\backslash (H,S)$ is defined as follows:%
\begin{align*}
(E\backslash (H,S))^{0}& =(E^{0}\backslash H)\cup \{v^{\prime }:v\in
B_{H}\backslash S\}; \\
(E\backslash (H,S))^{1}& =\{e\in E^{1}:r(e)\notin H\}\cup \{e^{\prime }:e\in
E^{1},r(e)\in B_{H}\backslash S\}.
\end{align*}%
Further, $r$ and $s$ are extended to $(E\backslash (H,S))^{0}$ by setting $%
s(e^{\prime })=s(e)$ and $r(e^{\prime })=r(e)^{\prime }$. Note that, in the
graph $E\backslash (H,S)$, the vertices $v^{\prime }$ are all sinks.

Theorem 5.7 of \cite{Tomforde 20} states that there is an epimorphism $\phi
:L_{K}(E)\rightarrow L_{K}(E\backslash (H,S))$ with $\ker \phi =$ $I_{(H,S)}$
and that $\phi (v^{H})=v^{\prime }$ for $v\in B_{H}\backslash S$. Thus $%
L_{K}(E)/I_{(H,S)}\cong L_{K}(E\backslash (H,S))$. This theorem has been
established in \cite{Tomforde 20} under the hypothesis that $E$ is a graph
with at most countably many vertices and edges; however, an examination of
the proof reveals that the countability condition on $E$ is not utilized. So
the Theorem 5.7 of \cite{Tomforde 20} holds for arbitrary graphs $E$.

\section{Generators of non-graded ideals of $L_{K}(E)$}

As noted earlier, Tomforde \cite{Tomforde 20} described a generating set for
the graded ideals of a Leavitt path algebra $L_{K}(E)$. In this section, as
a complement to Tomforde's theorem, we give an explicit description of a set
of generators for the non-graded ideals in $L_{K}(E)$. These generators are
then used in proving the main theorem of the next section.

We begin with the following useful result from \cite{ABCR}

\begin{theorem}
\label{IdealThm} Let $E$ be an arbitrary graph and $K$ be any field. Then
any non-zero ideal of the $L_{K}(E)$ is generated by elements of the form 
\begin{equation*}
(u+\tsum\limits_{i=1}^{k}k_{i}g^{r_{i}})(u-\tsum\limits_{e\in X}ee^{\ast })
\end{equation*}%
where $u\in E^{0}$, $k_{i}\in K$, $r_{i}$ are\ positive integers, $X$ is a
finite (possibly empty) proper subset of $s^{-1}(u)$ and, whenever $%
k_{i}\neq 0$ for some $i$, then $g$ is a unique cycle based at $u$.
\end{theorem}

The next Lemma is an extension of Lemma 3.3 in \cite{R} showing that ideals
of $L_{K}(E)$ containing no vertices are generated by a set of mutually
orthogonal polynomials over cycles.

\begin{lemma}
\label{Lermma3.3}Suppose $E$ is an arbitrary graph and $K$ is any field. If $%
N$ is a non-zero ideal of $L_{K}(E)$ which does not contain any vertices of $%
E$, then $N$ is a non-graded ideal and possesses a generating set of
mutually orthogonal generators of the form $y_{j}=(v_{j}+\tsum%
\limits_{i=1}^{n_{j}}k_{ji}g_{j}^{r_{i}})$ where (i) $g_{j}$ is a (unique)
cycle without exits based at the vertex $v_{j}$, (ii) $k_{ji}\in K$ with at
least one $k_{ji}\neq 0$ and $v_{r}\neq v_{s}$ ( so $y_{r}y_{s}=0$) if $%
r\neq s$.
\end{lemma}

\begin{proof}
Since $N$ is non-zero and since $H=N\cap E^{0}$ is the empty set, $N$ must
be a non-graded ideal, because if $N$ was a graded ideal, then $N$ must be $%
\{0\}$ since, by Tomforde \cite{Tomforde 20}, $N$ is generated by $H\cup
\lbrack \{v^{H}:v\in B_{H}\}\cap N]$ and $H$, $B_{H}$ are both empty sets.
From Theorem \ref{IdealThm}, we know that $N$ is generated by elements of
the form $y=(u+\tsum\limits_{i=1}^{n}k_{i}g^{r_{i}})(u-\tsum_{e\in
X}ee^{\ast })\neq 0$ where $g$ is a unique cycle in $E$ based at the vertex $%
u$ and where $X$ is a finite proper subset of $s^{-1}(u)$.

We wish to show that, for each such generator $y=(u+\tsum%
\limits_{i=1}^{n}k_{i}g^{r_{i}})(u-\tsum_{e\in X}ee^{\ast })$, the
corresponding cycle $g$ has no exits in $E$ and that $X$ must be an empty
set, so that $y=(u+\tsum\limits_{i=1}^{n}k_{i}g^{r_{i}})$. By hypothesis,
there is an $f\in s^{-1}(u)\backslash X$. Let $r(f)=w$. This $f$ must be the
initial edge of $g$. Because otherwise $f^{\ast }g=0$ and $(\tsum_{e\in
X}ee^{\ast })f=0$, and we obtain $f^{\ast }yf=f^{\ast
}(u+\tsum\limits_{i=1}^{n}k_{i}g^{r_{i}})f=f^{\ast }uf=r(f)=w\in N$, a
contradiction since $N$ contains no vertices. So we can write $g=f\alpha $
and let $h$ denote the cycle $\alpha f$ (based at $w$). Note that, in this
case, $f^{\ast }yf=f^{\ast
}(u+\tsum\limits_{i=1}^{n}k_{i}g^{r_{i}})(u-\tsum_{e\in X}ee^{\ast
})f=ff^{\ast }+f^{\ast
}\tsum\limits_{i=1}^{n}k_{i}g^{r_{i}}f=w+\tsum%
\limits_{i=1}^{n}k_{i}h^{r_{i}}\in N$. Then $\alpha ^{\ast
}(w+\tsum\limits_{i=1}^{n}k_{i}h^{r_{i}})\alpha
=u+\tsum\limits_{i=1}^{n}k_{i}g^{r_{i}}\in N$. Suppose, by way of
contradiction, there is an exit $e$ at a vertex $u^{\prime }$ on $g$. Let $%
\beta $ be the part of $g$ connecting $u$ to $u^{\prime }$ (where we take $%
\beta =u$ if $u^{\prime }=u$) and $\gamma $ be the part of $g$ from $%
u^{\prime }$ to $u$ ( so that $g=\beta \gamma $). Then, denoting the cycle $%
\gamma \beta $ (based at $u^{\prime }$) by $d$, we get $e^{\ast }\beta
^{\ast }(u+\tsum\limits_{i=1}^{n}k_{i}g^{r_{i}})\beta e=e^{\ast }(u^{\prime
}+\tsum\limits_{i=1}^{n}k_{i}d^{r_{i}})e=e^{\ast }e=r(e)\in N$, a
contradiction. Thus the cycle $g$ has no exits. In particular, $%
|s^{-1}(u)|=1 $ and this implies that $X$ must be an empty set, as $X$ is a
proper subset of $s^{-1}(u)$.

Thus the generators of $N$ are of the form $y=(u+\tsum%
\limits_{i=1}^{n}k_{i}g^{r_{i}})$. If there is another generator of $N$ of
the form $y^{\prime }=u+\tsum\limits_{i=1}^{n^{\prime }}k_{i}^{\prime
}(g^{\prime })^{s_{i}}$ with the same vertex $u$, then, by the uniqueness of 
$g$, $g^{\prime }=g$. Using the convention that $g^{0}=u$, we can write $%
y=p(g)$ and $y^{\prime }=q(g)$ where $p(x)=1+\tsum%
\limits_{i=1}^{n}k_{i}x^{r_{i}}$ and $p^{\prime
}(x)=1+\tsum\limits_{i=1}^{n^{\prime }}k_{i}^{\prime }x^{s_{i}}$ both
belonging to $K[x]$. If $d(x)$ is the gcd of $p(x)$ and $q(x)$ in $K[x]$,
then we can assume, without loss of generality, that $d(0)=1$. Moreover, we
\ can write $d(x)=a(x)p(x)+b(x)q(x)$ for suitable $a(x),b(x)\in K[x]$.
Clearly $d(g)=a(g)p(g)+b(g)q(g)\in I$ and we can then replace both $y=p(g)$
and $y^{\prime }=q(g)$ by $d(g)$. Iteration of this process guarantees that
different generators $y_{j}$ and $y_{k}$ involve different vertices $v_{j}$
and $v_{k}$ and so $y_{j}y_{k}=0=y_{k}y_{j}$ for $j\neq k$, resulting in a
mutually orthogonal set of generators for the ideal $I.$
\end{proof}

Since Condition (L) on a graph demands that cycles have exits, an immediate
consequence of Lemma \ref{Lermma3.3} is the following well-known result.

\begin{corollary}
\label{(L) implies vertex} \cite{Aranda Pino 10} Let $E$ be an arbitrary
graph. If $E$ satisfies Condition (L), then every non-zero two-sided ideal
of $L_{K}(E)$ contains a vertex.
\end{corollary}

The next theorem gives an explicit description of the generators of the
non-graded ideals of a Leavitt path algebra.

\begin{theorem}
\label{Generators of Ideal}Let $I$ be a non-zero ideal of $L_{K}(E)$ with $%
I\cap E^{0}=H$ and $S=\{v\in B_{H}:v^{H}\in I\}$. Then $I$ is generated by $%
H\cup \{v^{H}:v\in S\}\cup Y$, where $Y$ is a set of mutually orthogonal
elements of the form $(u+\tsum\limits_{i=1}^{n}k_{i}g^{r_{i}})$ in which (i) 
$g$ is a (unique) cycle without exits in $E^{0}\backslash H$ based at a
vertex $u$ in $E^{0}\backslash H$ and (ii) $k_{i}\in K$ with at least one $%
k_{i}\neq 0$. Moreover, $I$ is non-graded if and only if $Y$ is non-empty.
\end{theorem}

\begin{proof}
Let $J=I_{(H,S)}$ be the ideal of $L_{K}(E)$ generated by $H\cup
\{v^{H}:v\in S\}$. We may assume that $J\subsetneqq I$ since there is
nothing to prove if $I=J$. By Tomforde \cite{Tomforde 20}, $L_{K}(E)/J\cong
L_{K}(E\backslash (H,S))$. Identifying $L_{K}(E)/J$ \ with $%
L_{K}(E\backslash (H,S))$ via this isomorphism, we note that the non-zero
ideal $I/J$ contains no vertices of $E\backslash (H,S)$ and so by Lemma \ref%
{Lermma3.3}, $I/J$ is generated by elements of the form $(u+\tsum%
\limits_{i=1}^{n}k_{i}g^{r_{i}})$ where $g$ is a (unique) cycle without
exits in $E\backslash (H,S)$ based at a vertex $u\in (E\backslash
(H,S))^{0}=E^{0}\backslash H\cup \{v^{\prime }:v\in B_{H}\backslash S\}$ and 
$k_{i}\in K$ with at least one $k_{i}\neq 0$. It is then clear that the
ideal $I$ is generated by $H\cup \{v^{H}:v\in S\}\cup Y$, where $Y$ is the
set of mutually orthogonal elements of the form $y=(u+\tsum%
\limits_{i=1}^{n}k_{i}g^{r_{i}})$ where $g$ is a (unique) cycle without
exits in $E\backslash (H,S)$ based at a vertex $u\in (E\backslash (H,S))^{0}$
and $k_{i}\in K$ with at least one $k_{i}\neq 0$. Observe that since the $%
v^{\prime }\in (E\backslash (H,S))^{0}$ are all sinks, both $u$ and the
vertices on $g$ all belong to $E^{0}\backslash H$.
\end{proof}

Since Condition (K) on the graph $E$ implies that the set $Y$ in above
theorem must be empty, the following well-known result (see, for eg. \cite%
{Tomforde 20}) can be derived immediately from Theorem \ref{Generators of
Ideal}.

\begin{corollary}
Let $E$ be an arbitrary graph. Then $E$ satisfies Condition (K) if and only
if every ideal of $L_{K}(E)$ is graded.
\end{corollary}

\section{Finitely generated ideals of $L_{K}(E)$}

Here we show that any finitely generated two-sided ideal $I$ in a Leavitt
path algebra must be a principal ideal. The main idea of the proof is to
start with a generating set of the ideal $I$ as given Theorem \ref%
{Generators of Ideal} and to replace any finite subset of these generators
by an appropriate finite set of mutually orthogonal generators. The sum of
these orthogonal generators will be a desired single generator. As a
consequence, we derive that if $E$ is a finite graph, then the Leavitt path
algerbra $L_{K}(E)$ will be a two-sided principal ideal ring, that is, every
ideal of $L_{K}(E)$ will be a principal ideal.

\begin{theorem}
\label{Finited Gen Principal}Let $E$ be an arbitrary graph. Then every
finitely generated ideal of $L_{K}(E)$ is a principal ideal.
\end{theorem}

\begin{proof}
Suppose $E$ is an arbitrary graph and $I$ is an ideal of $L_{K}(E)$
generated by a finite set of elements $a_{1},...,a_{m}$ in $L_{K}(E)$. By
Theorem \ref{Generators of Ideal}, $I$ also has a generating set $H\cup
\{v^{H}:v\in S\}\cup Y$ where $H=I\cap E^{0}$, $S=\{v\in B_{H}:v^{H}\in I\}$
and $Y$ is a set of elements of the form $y=(u+\tsum%
\limits_{i=1}^{n}k_{i}g^{r_{i}})$ where $g$ is a (unique) cycle without
exits in $E\backslash (H,S)$ based at a vertex $u$ in $(E\backslash
(H,S))^{0}=E^{0}\backslash H$ and $k_{i}\in K$ with at least one $k_{i}\neq
0 $. Since each $a_{i}$ can be written as a finite sum of elements of the
form $\tsum\limits_{i=1}^{r_{1}}k_{i}\alpha _{i}\beta _{i}^{\ast }x\gamma
_{i}\delta _{i}^{\ast }$ where $x\in H\cup \{v^{H}:v\in S\}\cup Y$, we may
assume without loss of generality that the ideal $I$ is generated by a
finite set of elements $x_{1},...,x_{n}$ \ where $x_{i}\in H\cup
\{v^{H}:v\in S\}\cup Y$. We wish to re-choose the generators $x_{i}$ such
that for $i\neq j$, $x_{i}x_{j}=0$. This property clearly holds if $%
x_{i},x_{j}$ are different elements in either $H\cup \{v^{H}:v\in S\}$ or $%
H\cup Y$.

So we need only to consider the case when $x_{i}\in \{v^{H}:v\in S\}$ and $%
x_{j}\in Y$ with $x_{i}x_{j}\neq 0$ so that $x_{i}=v-\tsum\limits_{e\in
s^{-1}(v),r(e)\notin H}ee^{\ast }$ and $x_{j}=v+\tsum%
\limits_{i=1}^{n}k_{i}g^{r_{i}}$ where $v\in B_{H}$ and $g$ is a cycle
without exits based at $v$ in $E^{0}\backslash H$ . Since $g$ has no exits
in $E\backslash (H,S)$, $x_{i}=v-ee^{\ast }$ with $e$ the initial edge of $g$%
. Then $x_{i}x_{j}=(v-ee^{\ast
})(v+\tsum\limits_{i=1}^{n}k_{i}g^{r_{i}})=v-ee^{\ast
}+\tsum\limits_{i=1}^{n}k_{i}g^{r_{i}}-\tsum%
\limits_{i=1}^{n}k_{i}g^{r_{i}}=v-ee^{\ast }=x_{i}$ and so we remove $x_{i}$
from the list of generators \ of $I$. Repeating this process a finite number
of times, we obtain a finite set of generators $y_{1},...,y_{t}$ of the
ideal $I$ where $y_{i}=v_{i}y_{i}v_{i}$ for all $i$ and $v_{1},...,v_{t}$
are distinct vertices in $E$. An arbitrary element $z$ of $I$ will then be
of the form 
\begin{equation*}
z=\tsum\limits_{i=1}^{r_{1}}k_{1i}\alpha _{1i}\beta _{1i}^{\ast }y_{1}\gamma
_{1i}\delta _{1i}^{\ast }+\cdot \cdot \cdot
+\tsum\limits_{i=1}^{r_{t}}k_{ti}\alpha _{ti}\beta _{ti}^{\ast }y_{t}\gamma
_{ti}\delta _{ti}^{\ast }
\end{equation*}%
where, for $s=1,...,t$, $k_{si}\in K$, $\alpha _{si},\beta _{si},\gamma
_{si},\delta _{si}$ are all paths in $E$ for various $i$. Then 
\begin{equation*}
z=\tsum\limits_{i=1}^{r_{1}}k_{1i}\alpha _{1i}\beta _{1i}^{\ast }a\gamma
_{1i}\delta _{1i}^{\ast }+\cdot \cdot \cdot
+\tsum\limits_{i=1}^{r_{t}}k_{ti}\alpha _{ti}\beta _{ti}^{\ast }a\gamma
_{ti}\delta _{ti}^{\ast }
\end{equation*}%
where $a=y_{1}+...+y_{t}\in I$. This shows that $I$ is the principal ideal
generated by the element $a.$
\end{proof}

It was shown in \cite{C} that if $E$ is a finite graph, then every ideal of $%
L_{K}(E)$ is finitely generated. From Theorem \ref{Finited Gen Principal} we
then obtain the following stronger conclusion.

\begin{corollary}
Let $E$ be a finite graph. Then every ideal of $L_{K}(E)$ is a principal
ideal.
\end{corollary}


\begin{thebibliography}{9}
\bibitem{AA} G. Abrams and G. Aranda Pino, The Leavitt path algebra of a
graph, J Algebra, \textbf{293} (2005), 319 - 334.

\bibitem{ABCR} G. Abrams, J. P. Bell, P. Colak, and K. M. Rangaswamy,
Two-sided chain conditions on Leavitt path algebras over arbitrary graphs,
J. Alg. App., \textbf{11} (2012),

\bibitem{AMP} P. Ara, M.A. Moreno and E. Pardo, Non-stable K-theory for
graph algebras, Algebra and Representation Theory, \textbf{10} (2007), 157
-178.

\bibitem{Aranda Pino 10} G. Aranda Pino, D. Marti$^{\prime }$n Barquero, C.
Marti$^{\prime }$n Gonzalez, and M. Siles Molina, Socle theory for Leavitt
path algebras of arbitrary graphs, Rev. Mat. Iberoamericana 26(2) (2010)
611-638.

\bibitem{APS} G. Aranda Pino, E. Pardo and M. Siles Molina, Prime spectrum
and primitive Leavitt path algebras, Indiana Univ. Math. journal, \textbf{58}
(2009), 869 - 890.

\bibitem{C} P. Colak, Two-sided ideals in Leavitt path algebras, J. Alg.
App., \textbf{10 } (2011), 801 -

\bibitem{R} K.M. Rangaswamy, The theory of prime ideals of Leavitt path
algebras over arbitrary graphs, arXiv: 1106 4766v1 [Math.RA] 23 June 2011.

\bibitem{Tomforde 20} M. Tomforde, Uniqueness theorems and Ideal structure
of Leavitt path algebras, J. Algebra 318 (2007) 270 -299.
\end{thebibliography}
\end{document}